\documentclass{article}
\usepackage[a4paper,top=3cm,bottom=3.5cm,left=3.2cm,right=3.2cm,marginparwidth=1.75cm]{geometry}
\usepackage[hidelinks]{hyperref}
\usepackage{amsmath, amssymb, amsthm, bbm, comment, graphicx, cleveref, caption, siunitx, color, tikz, bigints, mathtools}

\newcommand{\ddp}[2]{\frac{\partial#1}{\partial#2}}
\newcommand{\dd}[2]{\frac{\mathrm{d}#1}{\mathrm{d}#2}}
\newcommand{\D}{\partial D}
\renewcommand{\S}{\mathcal{S}}
\newcommand{\K}{\mathcal{K}}
\newcommand{\outside}{\mathbb{R}^3\setminus \overline{D}}

\renewcommand*{\Re}{\operatorname{Re}}

\newcommand{\A}{\mathcal{A}}
\newcommand{\de}{\: \mathrm{d}}
\newcommand{\R}{\mathbb{R}}
\newcommand{\iu}{\mathrm{i}\mkern1mu}
\newcommand{\C}{\tilde{C}}

\newtheorem{theorem}{Theorem}

\newtheorem{lemma}[theorem]{Lemma}
\newtheorem{defn}[theorem]{Definition}

\newtheorem{remark}[theorem]{Remark}

\numberwithin{equation}{section}
\numberwithin{theorem}{section}

\usepackage[sort&compress,numbers]{natbib}

\title{Close-to-touching acoustic subwavelength resonators: eigenfrequency separation and gradient blow-up}
\date{}
\author{Habib Ammari\thanks{\footnotesize Department of Mathematics, ETH Z\"urich, R\"amistrasse 101, CH-8092 Z\"urich, Switzerland (habib.ammari@math.ethz.ch, bryn.davies@sam.math.ethz.ch).}\and Bryn Davies\footnotemark[1]  \and Sanghyeon Yu\thanks{Department of Mathematics, Korea University, Seoul 02841, S. Korea (sanghyeon\_yu@korea.ac.kr).}}

\begin{document}
\maketitle
\begin{abstract}
	In this paper, we study the behaviour of the coupled subwavelength resonant modes when two high-contrast acoustic resonators are brought close together. We consider the case of spherical resonators and use bispherical coordinates to derive explicit representations for the capacitance coefficients which, we show, capture the system's resonant behaviour at leading order. We prove that the pair of resonators has two subwavelength resonant modes whose frequencies have different leading-order asymptotic behaviour. We, also, derive estimates for the rate at which the gradient of the scattered pressure wave blows up as the resonators are brought together.
\end{abstract}
\vspace{0.5cm}
\noindent{\textbf{Mathematics Subject Classification (MSC2010):}} 35J05, 35C20, 35P20.
	
	\vspace{0.2cm}
	
	
	\begin{center}\rule{0.5\textwidth}{0.4pt}\end{center}

\section{Introduction}

Subwavelength acoustic resonators are compressible objects that experience resonant phenomena in response to wavelengths significantly greater than their size. This behaviour relies on the resonators being constructed from a material that has greatly different material parameters to the background medium. The classical example is an air bubble in water, in which case the subwavelength resonant mode is known as the \emph{Minnaert resonance} \cite{minnaert1933musical, ammari2018minnaert, devaud2008minnaert}. Thanks to their ability to interact with waves on subwavelength scales, structures made from subwavelength resonators (a type of \emph{metamaterial}) have been used for a wide variety of wave-guiding applications \cite{leroy2009design, ammari2017double, kushwaha1998sound, ammari2017sub, davies2019fully, davies2020hopf, ammari2019topological}.

In this paper, we wish to examine how the resonant modes of a pair of spherical resonators behave as they are brought close together. We will see that the leading-order behaviour of the resonant modes is determined by the so-called \emph{capacitance coefficients} \cite{ammari2017double}. These are well known in the setting of electrostatics and can be calculated explicitly when the resonators are spherical. This was first realised by Maxwell, who derived the formula using the method of image charges, but we favour the approach conceived by Jeffrey, which relies on expanding solutions using bispherical coordinates \cite{jeffery1912form}.

The analysis performed here will rely on using layer potentials to represent solutions, enabling us to perform an asymptotic analysis in terms of the material contrast \cite{ammari2018minnaert, ammari2017double}. We will show that the two subwavelength resonant frequencies have different asymptotic behaviours, which can be neatly expressed if the separation distance is chosen as a function of the material contrast. This has significant implications for the design of acoustic metamaterials with multi-frequency or broadband functionality. We will then examine how the eigenmodes behave as the resonators are brought together. In particular, we study the extent to which the gradient of the acoustic pressure between the resonators blows up as they are brought together.

Similar analyses (of close-to-touching material inclusions) have been performed in several other settings. In the context of electrostatics \cite{poladian1989asymptotic,  mcphedran1988asymptotic, ammari2007optimal, yun2007estimates, lim2009blow,  bao2009gradient, lekner2011near, lim2014asymptotic, gorb2015singular, kim2019electric}  and linear elasticity \cite{kang2019quantitative, ammari2013spectral, bao2015gradient, bao2017gradient, lim2014asymptotic_elas}, it has been shown that the electric field or the stress field blows up as the two inclusions get closer, provided the material parameters of the inclusions are infinite or zero. However, these field enhancement phenomena are not related to resonances. 
In the plasmonic case, where electromagnetic inclusions have negative permittivity and support subwavelength resonances called surface plasmons, the close-to-touching interaction has been studied in \cite{yu2019hybridization, yu2018plasmonic, bonnetier2013spectrum, mcphedran1981electrostatic,  pendry2012transformation, romero2006plasmons, khattak2019linking, hooshmand2019collective, ammari2019shape}. The subwavelength resonance studied in this article is a quite different phenomenon from the behaviour of surface plasmons. Firstly, we study structures with positive and highly contrasting material parameters and, secondly, high-contrast acoustic resonators exhibit both monopolar and dipolar resonances while plasmonic inclusions support only dipolar resonant modes.

Since many of the results concerning the resonant modes and frequencies can be expressed as concise formulas when the resonators are identical, we summarise the results for this special case in \Cref{sec:sym} before proving the general versions in \Cref{sec:asymm,sec:eigenmodes,sec:scattering}. Finally, in \Cref{sec:scattering} we demonstrate the value of these results by expressing the scattered solution in terms of the resonant modes.

\section{Preliminaries}
\subsection{Asymptotic notation}
We will use the following two pieces of notation in this work.
\begin{defn}
Let $f$ be a real- or complex-valued function and $g$ a real-valued function that is strictly positive in a neighbourhood of $x_0$. We write that
\begin{equation*}
    f(x)=O(g(x)) \text{ as } x\to x_0,
\end{equation*}
if and only if there exists some positive constant $M$ such that $|f(x)|\leq M g(x)$ for all $x$ such that $x-x_0$ is sufficiently small.
\end{defn}

\begin{defn}
Let $f$ and $g$ be real-valued functions which are strictly positive in a neighbourhood of $x_0$. We write that
\begin{equation*}
    f(x)\sim g(x) \text{ as } x\to x_0,
\end{equation*}
if and only if both $f(x)=O(g(x))$ and $g(x)=O(f(x))$ as $x\to x_0$.
\end{defn}

\subsection{Helmholtz formulation}

We study the Helmholtz problem which describes how a time-harmonic plane wave is scattered by the high-contrast structure. We consider a homogeneous background medium with density $\rho$ and bulk modulus $\kappa$. We study the effect of scattering by a pair of spherical inclusions, $D_1$ and $D_2$, with radii $r_1$ and $r_2$ and separation distance $\epsilon$ (so that their centres are separated by $r_1+r_2+\epsilon$). We use $\rho_b$ and $\kappa_b$ for the density and bulk modulus of the resonators' interior and introduce the auxiliary parameters
\begin{equation*}
v=\sqrt{\frac{\kappa}{\rho}}, \quad v_b=\sqrt{\frac{\kappa_b}{\rho_b}}, \quad k=\frac{\omega}{v}, \quad k_b=\frac{\omega}{v_b},
\end{equation*}
which are the wave speeds and wavenumbers in $\outside$ and in $D$, respectively. We, finally, introduce the two dimensionless contrast parameters
\begin{equation}
\delta = \frac{\rho_b}{\rho}, \quad \tau = \frac{k_b}{k} = \frac{v_b}{v} = \sqrt{\frac{\rho \kappa_b}{\rho_b \kappa}}.
\end{equation}

If we use the subscripts $+$ and $-$ to denote evaluation from outside and inside $\D$ respectively, then the acoustic pressure $u$ produced by the scattering of an incoming plane wave  $u^{in}$ satisfies
\begin{equation} \label{eq:helmholtz_equation}
\begin{cases}
\left( \Delta + k^2 \right) u = 0, & \text{in } \outside, \\
\left( \Delta + k_b^2 \right) u = 0, & \text{in } D, \\
u_+ - u_- = 0, & \text{on } \D,\\
\delta \ddp{u}{\nu}\big|_+ -  \ddp{u}{\nu}\big|_- = 0, & \text{on } \D, 
\end{cases}
\end{equation}
along with the Sommerfeld radiation condition, namely,
\begin{equation}
\left(\ddp{}{|x|} - i k\right) (u-u^{in})(x) = O(|x|^{-2}), \quad \mbox{as }|x|\rightarrow \infty.
\end{equation}
We assume that $v$, $v_b$, $\tau$, $r_1$ and $r_2$ are all $O(1)$. On the other hand, we assume that there is a large contrast between the densities, so that
\begin{equation}
0<\delta \ll 1.
\end{equation}
A classic example of material inclusions satisfying these assumptions is a collection of air bubbles in water, often known as Minnaert bubbles \cite{minnaert1933musical}, in which case we have $\delta\approx10^{-3}$.

We choose the separation distance $\epsilon$ as a function of $\delta$ and will perform an asymptotic analysis in terms of $\delta$. We choose $\epsilon$ to be such that, for some $0<\beta<1$,
\begin{equation} \label{assum_epsilon}
\epsilon\sim e^{-1/\delta^{1-\beta}} \text{ as } \delta\to0.
\end{equation}
As we will see shortly, with $\epsilon$ chosen to be in this regime the subwavelength resonant frequencies are both well behaved (\emph{i.e.} $\omega=\omega(\delta)\to0$ as $\delta\to0$) and we can compute asymptotic expansions in terms of $\delta$.

\subsection{Layer potentials}

Let $D\subset \R^3$ be the union of the two disjoint spheres $D_1$ and $D_2$. Let $G^k$  be the (outgoing) Helmholtz Green's function
\begin{equation} \label{eq:greens_fn}
G^k(x,y) := -\frac{e^{\iu k|x-y|}}{4\pi|x-y|}, \quad x,y \in \R^3, k\geq 0,
\end{equation}
and $\S_{D}^k: L^2(\partial D) \rightarrow H_{\textrm{loc}}^1(\R^3)$ the corresponding \emph{single layer potential} \cite{colton1983integral, ammari2009layer}, defined by
\begin{equation}
\S_D^k[\phi](x) := \int_{\D} G^k(x,y)\phi(y) \de \sigma(y), \quad x \in \R^3, \, \phi\in L^2(\partial D).
\end{equation}
Here, $H_{\textrm{loc}}^1(\R^3)$  is the usual Sobolev space. We also define the \emph{Neumann--Poincar\'e} operator $\K_D^{k,*}: L^2(\partial D) \rightarrow L^2(\D)$ by
\begin{equation}
\K_D^{k,*}[\phi](x) := \int_{\partial D} \frac{\partial }{\partial \nu_x}G^k(x,y) \phi(y) \de \sigma(y), \quad x \in \partial D,
\end{equation}
where $\partial/\partial \nu_x$ denotes the outward normal derivative at $x\in\D$.

The solutions to \eqref{eq:helmholtz_equation} can be represented as \cite{ammari2018mathematical}
\begin{equation} \label{eq:layer_potential_representation}
u = \begin{cases}
u^{in}(x)+\S_{D}^k[\psi](x), & x\in\outside,\\
\S_{D}^{k_b}[\phi](x), & x\in D,
\end{cases}
\end{equation} 
for some surface potentials $(\phi,\psi)\in L^2(\D)\times L^2(\D)$, which must be chosen so that $u$ satisfies the two transmission conditions across $\D$. This is equivalent to satisfying (see \emph{e.g.} \cite{ammari2004boundary, colton1983integral, ammari2018mathematical} for details)
\begin{equation} \label{eq:A_matrix_equation}
\A(\omega,\delta)\begin{pmatrix} \phi \\ \psi \end{pmatrix} = \begin{pmatrix} u^{in} \\ \delta\ddp{u^{in}}{\nu_x}\end{pmatrix},
\end{equation}
where
\begin{equation*} \label{eq:A_matrix_defn}
\A(\omega,\delta):=
\left[ {\begin{array}{cc}
	\S_D^{k_b} & -\S_D^k \\
	-\frac{1}{2}I+\K_D^{k_b,*} & -\delta(\frac{1}{2}I+\K_D^{k,*}) \\
	\end{array} } \right],
\end{equation*}
and $I$ is the identity operator on $L^2(\D)$.

Our analysis of \eqref{eq:A_matrix_equation} will be asymptotic using the fact that $\delta\ll1$, by assumption, and we are interested in subwavelength resonant modes for which $\omega\ll1$. Using the exponential power series we can derive an expansion for $\S_D^k$, given by
\begin{equation} \label{eq:S_series}
\S_D^k= \S_{D} + \sum_{n=1}^{\infty} k^n \S_{D,n},
\end{equation}
where, for $n=0,1,2,\dots$,
\begin{equation*}
\S_{D,n}[\phi](x):=-\frac{\iu^n}{4\pi n!} \int_{\D} |x-y|^{n-1}\phi(y) \de\sigma(y), \quad x \in \R^3, \, \phi\in L^2(\partial D),
\end{equation*}
and $\S_D:=\S_{D,0}$ is the Laplace single layer potential. It is well known that $\S_D: L^2(\D) \rightarrow H^1(\D)$ is invertible \cite{ammari2018mathematical}. Similarly, for $\K_D^{k,*}$ we have that 
\begin{equation} \label{eq:K_series}
\K_D^{k,*}=\K_D^*+\sum_{n=1}^{\infty} k^n\K_{D,n},
\end{equation}
where, for $n=0,1,2,\dots$,
\begin{equation*}
\K_{D,n}[\phi](x):=-\frac{\iu^n(n-1)}{4\pi n!} \int_{\D} |x-y|^{n-3}(x-y)\cdot \nu_x \,\phi(y) \de\sigma(y), \quad x \in \R^3, \, \phi\in L^2(\partial D),
\end{equation*}
and $\K_D^*:=\K_{D,0}$ is the Neumann--Poincar\'e operator corresponding to the Laplace equation.

The kernels of the integral operators $\S_{D,n}$ and $\K_{D,n}$ for $n\geq1$ are bounded as $|x-y|\to0$. Conversely, the kernels of $\S_D$ and $\K_D$ have singularities in the $\epsilon\to0$ limit. Thus, for small $k$, the leading order terms in \eqref{eq:S_series} and \eqref{eq:K_series} dominate even for small $\epsilon>0$. This allows us to write that, as $k,\epsilon\to0$, $\S_D^k= \S_{D} + O(k)$ and $\K_D^{k,*}= \K_{D}^* + O(k)$ in the relevant operator norms. This is made precise by the following lemma (\emph{cf.} \cite{colton1983integral, ammari2018mathematical}).

\begin{lemma} \label{lem:series_conv}
	The norms $\| \S_{D,n} \|_{\mathcal{B}(L^2(\D), H^1(\D))}$ and $\| \K_{D,n} \|_{\mathcal{B}(L^2(\D), L^2(\D))}$ are uniformly bounded for $n\geq1$ and $0<\epsilon\ll1$. Moreover, the series in \eqref{eq:S_series} and \eqref{eq:K_series} are uniformly convergent for $\epsilon>0$, in $\mathcal{B}(L^2(\D), H^1(\D))$ and $\mathcal{B}(L^2(\D), L^2(\D))$ respectively.
\end{lemma}

\subsection{Resonant frequencies}

In light of the representation \eqref{eq:layer_potential_representation}, we can define the notion of resonance to be the existence of a non-trivial solution when the incoming field $u^{in}$ is zero.

\begin{defn} \label{defn:resonance}
	For a fixed $\delta$ we define a resonant frequency (or eigenfrequency) to be $\omega\in\mathbb{C}$ such that there exists a non-trivial solution to
	\begin{equation} \label{eq:res}
	\A(\omega,\delta)
	\begin{pmatrix}
	\phi \\ \psi
	\end{pmatrix}
	=
	\begin{pmatrix}
	0 \\ 0
	\end{pmatrix},
	\end{equation}
	where $\A(\omega,\delta)$ is defined in \eqref{eq:A_matrix_defn}.
	For each resonant frequency $\omega$ we define the corresponding resonant mode (or eigenmode) as
	\begin{equation} \label{eq:eigenmode_representation}
	u = \begin{cases}
	\S_{D}^{k_b}[\phi](x), & x\in D, \\
	\S_{D}^k[\psi](x), & x\in\outside.
	\end{cases}
	\end{equation}
\end{defn}

\begin{remark}
The resonant modes \eqref{eq:eigenmode_representation} are determined only up to normalisation. In \Cref{sec:eigenmodes,sec:scattering} we will choose the normalisation to be such that $u_n\sim1$ on $\D$ for all small $\delta$ and $\epsilon$.
\end{remark}

\begin{defn}
	We define a subwavelength resonant frequency to be a resonant frequency $\omega=\omega(\delta)$ such that $\omega(0)=0$ and $\omega$ depends on $\delta$ continuously.
\end{defn}

\begin{lemma} \label{lem:two_modes}
	There exist two subwavelength resonant modes, $u_1$ and $u_2$, with associated resonant frequencies $\omega_1$ and $\omega_2$ with positive real part, labelled such that $\Re(\omega_1)<\Re(\omega_2)$.
\end{lemma}
\begin{proof}
	Consider the operator corresponding to $\delta=0$ and $\omega=0$:
	\begin{equation*}
	\A(0,0)=\left[ {\begin{array}{cc}
		\S_D & -\S_D \\
		-\frac{1}{2}I+\K_D^{*} & 0 \\
		\end{array} } \right].
	\end{equation*}
	Since $\S_D$ is invertible, $\dim\ker\A(0,0)=\dim\ker\left(-\frac{1}{2}I+\K_D^{*}\right)$. We can show (\emph{e.g.} the arguments in Lemma~2.12 of \cite{davies2019fully}) that 
	\begin{equation*}
	\left\{\S_D^{-1}[\mathcal{X}_{\D_1}],\,\S_D^{-1}[\mathcal{X}_{\D_2}]\right\},
	\end{equation*}
	is a basis for $\ker\left(-\frac{1}{2}I+\K_D^{*}\right)$. Then, by the theory of Gohberg and Sigal \cite{gohberg2009holomorphic, ammari2009layer}, we have that there exist two subwavelength resonant modes $u_1$ and $u_2$, at leading order.
\end{proof}

\begin{remark}
	We will see, shortly, that each resonant mode has two resonant frequencies associated to it with real parts that differ in sign. We will use the notation $\omega_n$ to denote the resonant frequency associated to $u_n$ that has positive real part. As we will see in the proof of \Cref{lem:res_asym}, $-\Re(\omega_n)$ is also a resonant frequency associated to the mode $u_n$, up to an error of order $O(\delta)$.
\end{remark}

\subsection{Capacitance coefficients}

Let $\psi_1,\psi_2\in L^2(\D)$ be given by
\begin{equation} \label{eq:psi_defs}
\S_D[\psi_1]=\begin{cases}
1&\text{on } \D_1,\\
0&\text{on } \D_2,
\end{cases}
\qquad
\S_D[\psi_2]=\begin{cases}
0&\text{on } \D_1,\\
1&\text{on } \D_2.
\end{cases}
\end{equation}
We can show (as in the proof of \Cref{lem:two_modes}) that
\begin{equation} \label{kernel_basis}
\ker\left(-\frac{1}{2}I+\K_D^*\right)=\text{span}\{\psi_1,\psi_2\}.
\end{equation}
We then define the \emph{capacitance matrix} $C=(C_{ij})$ as
\begin{equation}
C_{ij}:=-\int_{\D_i} \psi_j \de \sigma, \quad i,j=1,2.
\end{equation}

\begin{figure}
    \centering
    \captionsetup{width=.6\linewidth}
	\includegraphics[width=.7\linewidth]{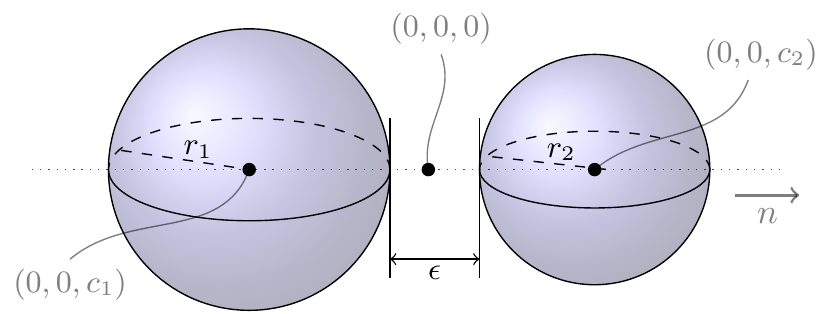}
    \caption{Two close-to-touching spheres, annotated with the coordinate system outlined in \Cref{sec:coordinates}.}
    \label{fig:coordinates}
\end{figure}

\subsection{Coordinate system} \label{sec:coordinates}

The Helmholtz problem \eqref{eq:helmholtz_equation} is invariant under translations and rotations so we are free to choose the coordinate axes. Let $R_j$ be the reflection with respect to $\partial D_j$ and let $p_1$ and $p_2$ be the unique fixed points of the combined reflections $R_1\circ R_2$ and $R_2\circ R_1$, respectively. Let $n$ be the unit vector in the direction of $p_2-p_1$. We will make use of the Cartesian coordinate system $(x_1,x_2,x_3)$ defined to be such that $p=(p_1+p_2)/2$ is the origin and the $x_3$-axis is parallel to the unit vector $n$. Then one can see that \cite{kang2019quantitative}
\begin{equation}
p_1 = (0,0,-\alpha)\quad \mbox{and} \quad p_2 = (0,0,\alpha),
\end{equation}
where
\begin{equation}
\alpha:=\frac{\sqrt{\epsilon(2r_1+\epsilon)(2r_2+\epsilon)(2r_1+2r_2+\epsilon)}}{2(r_1+r_2+\epsilon)}.
\end{equation}
Moreover, the sphere $D_i$ is centered at $(0,0,c_i)$ where
\begin{equation} \label{eq:centres}
c_i=(-1)^i \sqrt{r_i^2+\alpha^2}.
\end{equation}
This is depicted in \Cref{fig:coordinates}. This Cartesian coordinate system is chosen so that we can define a bispherical coordinate system \eqref{def:bispherical_coordinates} such that the boundaries of the two resonators are convenient level sets.

\section{The special case of identical spheres} \label{sec:sym}

In this section, we summarise the results in the case that the two spheres have the same radius, which we denote by $r$. These results are all special cases of those derived in the rest of this paper. Firstly, the resonant frequencies are given, in terms of the capacitance coefficients, by
\begin{equation} \label{eq:sym_resonances}
\begin{split}
	\omega_1&=\sqrt{\delta \frac{3v_b^2}{4\pi r^3}(C_{11}+C_{12})}+O(\delta),\\
	\omega_2&=\sqrt{\delta \frac{3v_b^2}{4\pi r^3}(C_{11}-C_{12})}+O(\delta).
\end{split}
\end{equation}
Further to this, since $D_1$ and $D_2$ are spherical we can derive explicit expressions for the capacitance coefficients. In the case that the resonators are identical, the capacitance coefficients are given by
\begin{equation} \label{eq:sym_cap}
\begin{split}
C_{11}&=C_{22}=8\pi\tilde{\alpha}\sum_{n=0}^{\infty} \frac{e^{(2n+1)\xi_0}}{e^{2(2n+1)\xi_0}-1}, \\
C_{12}&=C_{21}=-8\pi\tilde{\alpha}\sum_{n=0}^{\infty}\frac{1}{e^{2(2n+1)\xi_0}-1},
\end{split}
\end{equation}
where
\begin{equation*}
\tilde{\alpha}:=\sqrt{\epsilon(r+\epsilon/4)}, \qquad
\xi_0:=\sinh^{-1}\left(\frac{\tilde{\alpha}}{r}\right).
\end{equation*}
From \cite{lekner2011near}, we know the asymptotic behaviour of the series in \eqref{eq:sym_cap} as $\xi_0\to0$, from which we can see that as $\epsilon\to0$,
\begin{equation} \label{eq:cap_sym_asym}
\begin{split}
C_{11}&=2\pi \frac{\tilde{\alpha}}{\xi_0}\left[\gamma+2\log2+\log\left(\sqrt{r}\right)-\log\left(\sqrt{\epsilon}\right)\right]+O(\epsilon),\\
C_{12}&=-2\pi \frac{\tilde{\alpha}}{\xi_0}\left[\gamma+\log\left(\sqrt{r}\right)-\log\left(\sqrt{\epsilon}\right)\right]+O(\epsilon),
\end{split}
\end{equation}
where $\gamma\approx0.5772\dots$ is the Euler–Mascheroni constant.

Combining \eqref{eq:sym_resonances} and \eqref{eq:cap_sym_asym} we reach the fact that the resonant frequencies are given, as $\delta\to0$, by
\begin{equation} \label{eq:res_identical}
\begin{split}
\omega_1&=\sqrt{\delta \frac{3v_b^2\log2}{r^2}}+O\left(\delta\right),\\
\omega_2&=\sqrt{\delta \frac{3v_b^2}{2r^2} \left(\log\left(\frac{r}{\epsilon}\right)+2\gamma+2\log2\right)}+O\left(\sqrt{\delta}\right).
\end{split}
\end{equation}
Thus, the choice of $\epsilon\sim e^{-1/\delta^{1-\beta}}$, where $0<\beta<1$, means that as $\delta\to0$ we have that  $\omega_1\sim\sqrt{\delta}$ and $\omega_2\sim\delta^{\beta/2}$.

The two resonant modes, $u_1$ and $u_2$, correspond to the two resonators oscillating in phase and in antiphase with one another, respectively. Since the eigenmode $u_2$ has different signs on the two resonators, $\nabla u_2$ will blow up as the two resonators are brought together. Conversely, $u_1$ takes the same value on the two resonators so there will not be a singularity in the gradient. In particular, if we normalise the eigenmodes so that for any $x\in\D$
\begin{equation}
\lim_{\delta\to0} |u_1(x)|\sim1,\qquad \lim_{\delta\to0} |u_2(x)|\sim1,
\end{equation}
then the choice of $\epsilon$ to satisfy the regime $\epsilon\sim e^{-1/\delta^{1-\beta}}$ means that the maximal gradient of each eigenmode has the asymptotic behaviour, as $\delta\to0$,
\begin{equation} \label{eq:sym_mode_blowup}
\max_{x\in\outside}|\nabla u_1(x)|\sim 1, \qquad
\max_{x\in\outside}|\nabla u_2(x)|\sim \frac{1}{\epsilon}.
\end{equation}
By decomposing the scattered field into the two resonant modes, we can use \eqref{eq:sym_mode_blowup} to understand the singular behaviour exhibited by the acoustic pressure. The solution $u$ to the scattering problem \eqref{eq:helmholtz_equation} with incoming plane wave $u^{in}$ with frequency $\omega\ll1$ is given, for $x\in\outside$, by
\begin{equation}
u(x)=u^{in}(x)-\S_D^k\left[\S_D^{-1}[u^{in}]\right](x)+au_1(x)+bu_2(x)+O(\omega),
\end{equation}
where the coefficients $a$ and $b$ are given, as $\delta\to0$, by
\begin{align*}
a&= \frac{\delta}{\omega^2-\omega_1^2}\frac{v_b^2}{|D|}\int_{\D} \S_D^{-1}[ u^{in}] \de\sigma +O(\delta^{2-\beta}+\delta^{1-\beta}\omega^2+\omega^3),\\
b&= -\frac{\delta}{\omega^2-\omega_2^2}\frac{v_b^2}{|D|}\left(\int_{\D_1} \S_D^{-1}[ u^{in}] \de\sigma -\int_{\D_2} \S_D^{-1}[ u^{in}] \de \sigma\right) +O(\delta^{2-\beta}+\delta^{1-\beta}\omega^2+\omega^3),
\end{align*}
with $|D|$ being the volume of $D=D_1\cup D_2$.

\section{Resonant modes} \label{sec:asymm}

We now derive results analogous to those in \Cref{sec:sym} for the more general case where $D_1$ and $D_2$ are arbitrarily sized spheres with respective radii $r_1$ and $r_2$. We only require that $r_1,r_2=O(1)$. In the case of non-identical spheres it is convenient to define the \emph{rescaled capacitance matrix} $\tilde{C}=(\tilde{C}_{ij})$ as
\begin{equation} \label{eq:cap_rescaled}
\tilde{C}_{ij}:=\frac{1}{|D_i|}C_{ij},
\end{equation}
where $|D_i|=4\pi r_i^3/3$ is the volume of the sphere $D_i$. The resonant frequencies are determined by the eigenvalues of the rescaled capacitance matrix.

\begin{lemma} \label{lem:res_asym}
	The subwavelength resonant frequencies of two resonators $D_1$ and $D_2$ are given, as $\delta\to0$, for $n=1,2$, by
	\begin{equation*}
	\omega_n=\sqrt{\delta v_b^2 \lambda_n}+O(\delta),
	\end{equation*}
	where $\lambda_1$, $\lambda_2$ are the eigenvalues of the rescaled capacitance matrix $\tilde{C}$, defined in \eqref{eq:cap_rescaled}.
\end{lemma}
\begin{proof}
Suppose that $(\phi,\psi)$ is a solution to \eqref{eq:res} for small $\omega=\omega(\delta)$. From the asymptotic expansions \eqref{eq:S_series} and \eqref{eq:K_series} we have that 
\begin{align} 
\S_D[\phi-\psi]+k_b\S_{D,1}[\phi]-k\S_{D,1}[\psi]&=O(\omega^2),\label{eq:first}\\
\left(-\frac{1}{2}I+\K_D^*+k_b^2\K_{D,2}\right)[\phi]-\delta\left(\frac{1}{2}I+\K_D^*\right)[\psi]&=O(\delta\omega+\omega^3) \label{eq:second}.
\end{align}
From the first equation \eqref{eq:first} and the fact that $\S_D$ is invertible we can see that $\phi=\psi+O(\omega)$. We recall, \emph{e.g.} from Lemma~2.1 of \cite{ammari2017double}, that for any $\varphi\in L^2(\D)$ we have
\begin{equation} \label{k_properties}
\begin{split}
\int_{\D_i}\left(-\frac{1}{2}I+\K_D^{*}\right)[\varphi]\de\sigma=0,
\qquad&\int_{\D_i}\left(\frac{1}{2}I+\K_D^{*}\right)[\varphi]\de\sigma=\int_{\D_i}\varphi\de\sigma,\\
\int_{\D_i} \K_{D,2}[\varphi]\de\sigma&=-\int_{D_i}\S_D[\varphi]\de x,
\end{split}
\end{equation}
for $i=1,2$. Integrating \eqref{eq:second} over $\D_i$, for $i=1,2$, and using \eqref{k_properties} gives us that
\begin{equation} \label{eq:313}
-k_b^2\int_{D_i}\S_D[\psi]\de\sigma -\delta\int_{\D_i}\psi\de\sigma =O(\delta\omega+\omega^3).
\end{equation}

At leading order, \eqref{eq:second} says that $\left(-\frac{1}{2}I+\K_D^{*}\right)[\psi]=0$ so, in light of \eqref{kernel_basis}, the solution can be written as
\begin{equation} \label{eq:psi_basis}
\psi=a_1\psi_1+a_2\psi_2+O(\omega^2+\delta),
\end{equation}
for constants $a_1,a_2=O(1)$. Making this substitution into \eqref{eq:313} we reach, up to an error of order $O(\omega^3+\delta \omega)$, the eigenvalue problem
\begin{equation}
\tilde{C}\begin{pmatrix}a_1\\a_2\end{pmatrix}
=\frac{k_b^2}{\delta}\begin{pmatrix}a_1\\a_2\end{pmatrix}.
\end{equation}
\end{proof}

\begin{remark}
It is important, at this point, to highlight the fact that the resonant frequencies $\omega_1$ and $\omega_2$ are not real valued. Since we are studying resonators in an unbounded domain, energy is lost to the far field meaning that the resonant frequencies have negative imaginary parts \cite{ammari2018minnaert, ammari2017double, davies2019fully}. The leading order terms in the expansions for $\omega_1$ and $\omega_2$ (given in \Cref{lem:res_asym}) are real valued and the imaginary parts will appear in higher-order terms in the expansion. Since only the leading order terms in the asymptotic expansion \eqref{eq:S_series} and \eqref{eq:K_series} have singularities as the resonators are moved close together, it is not enlightening to study higher-order expansions in this work.
\end{remark}

By elementary linear algebra we have that the eigenvalues of $\C$ are given by
\begin{equation} \label{eq:eigenvalues}
\lambda_n=\frac{1}{2}\left(\C_{11}+\C_{22}+(-1)^n\sqrt{(\C_{11}-\C_{22})^2+4\C_{12}\C_{21}}\right),
\end{equation}
for $n=1,2$.  From \eqref{eq:eigenvalues}, finding the resonant frequencies (at leading order) has been reduced to finding expressions for the capacitance coefficients.

\begin{lemma} \label{lem:cap_coeffs}
	In the case that $D_1$ and $D_2$ are spheres of radius $r_1$ and $r_2$, respectively, and are separated by a distance $\epsilon$ the capacitance coefficients are given by
	\begin{equation*}
	C_{11}=8\pi\alpha\sum_{n=0}^{\infty} \frac{e^{(2n+1)\xi_2}}{e^{(2n+1)(\xi_1+\xi_2)}-1}, \qquad C_{22}=8\pi\alpha\sum_{n=0}^{\infty} \frac{e^{(2n+1)\xi_1}}{e^{(2n+1)(\xi_1+\xi_2)}-1},
	\end{equation*}
	\begin{equation*}
	C_{12}=C_{21}=-8\pi\alpha\sum_{n=0}^{\infty}\frac{1}{e^{(2n+1)(\xi_1+\xi_2)}-1},
	\end{equation*}
	where
	\begin{equation*}
	\xi_i:=\sinh^{-1}\left(\frac{\alpha}{r_i}\right).
	\end{equation*}
\end{lemma}
\begin{proof}
Let $V_j:=\S_D[\psi_j]$ be defined as the extension of \eqref{eq:psi_defs} to all of $\outside$, for $j=1,2$. Then $V_j$ is the unique solution to the problem
\begin{equation} \label{eq:capacitance_equation}
\begin{cases}
\Delta V_j = 0, & \text{in } \outside, \\
V_j=\delta_{ij}, & \text{on } \D_i,\\
V_j(x)=O\left(\tfrac{1}{|x|}\right), & \text{as } |x|\to\infty.
\end{cases}
\end{equation}
By recalling the transmission conditions for the single layer potential on $\D$ \cite{ammari2018mathematical}, in particular the fact that for any $\varphi\in L^2(\D)$
\begin{equation*}
\ddp{}{\nu}\S_D[\varphi]|_\pm=\left(\pm\frac{1}{2} I+\K_D^*\right)[\varphi],
\end{equation*}
on $\D$ and using \eqref{k_properties} we can write the capacitance coefficients in the form
\begin{equation} \label{eq:cap_reform}
C_{ij}=-\int_{\D_i} \ddp{V_j}{\nu}\bigg|_+ \de \sigma, \quad i,j=1,2.
\end{equation}

We will find expressions for $V_i$ using bispherical coordinates. Recall the Cartesian coordinate system $(x_1,x_2,x_3)$ from \Cref{sec:coordinates}, which is such that $p_1=(0,0,-\alpha)$ and $p_2=(0,0,\alpha)$ are the fixed points of the combined reflections in $\D_1$ and $\D_2$, where $\alpha$ is given by
\begin{equation*}
\alpha=\frac{\sqrt{\epsilon(2r_1+\epsilon)(2r_2+\epsilon)(2r_1+2r_2+\epsilon)}}{2(r_1+r_2+\epsilon)}. 
\end{equation*}
We then introduce a bispherical coordinate system $(\xi,\theta,\varphi)$ which is related to $(x_1,x_2,x_3)$ by
\begin{equation} \label{def:bispherical_coordinates}
x_1=\frac{\alpha\sin\theta\cos\varphi}{\cosh\xi-\cos\theta}\,,\quad 
x_2=\frac{\alpha\sin\theta\sin\varphi}{\cosh\xi-\cos\theta}\,,\quad 
x_3=\frac{\alpha\sinh\xi}{\cosh\xi-\cos\theta}\,,
\end{equation}
and is chosen to satisfy $-\infty<\xi<\infty$, $0\leq\theta<\pi$ and $0\leq\varphi<2\pi$. The reason for this choice of coordinate system is that $\D_1$ and $\D_2$ are given by the level sets
\begin{equation}
\D_1=\{\xi=-\xi_1\},\qquad \D_2=\{\xi=\xi_2\},
\end{equation}
where $\xi_1$, $\xi_2$ are positive constants given by
\begin{equation}
\xi_j:=\sinh^{-1}\left(\frac{\alpha}{r_i}\right).
\end{equation}

We now show that 
\begin{equation} \label{eq:V_sol}
V_j(\xi,\theta,\varphi)=\sqrt{2}\sqrt{\cosh\xi-\cos\theta} \sum_{n=0}^{\infty}\left(A_n^j e^{(n+\frac{1}{2})\xi} + B_n^j e^{-(n+\frac{1}{2})\xi}\right)P_n(\cos\theta),
\end{equation}
where $P_n$ are the Legendre polynomials and
\begin{align*}
A_n^1&= \frac{1}{1-e^{(2n+1)(\xi_1+\xi_2)}},
&B_n^1&= -\frac{e^{(2n+1)\xi_2}}{1-e^{(2n+1)(\xi_1+\xi_2)}},\\
A_n^2&= -\frac{e^{(2n+1)\xi_1}}{1-e^{(2n+1)(\xi_1+\xi_2)}},
&B_n^2&= \frac{1}{1-e^{(2n+1)(\xi_1+\xi_2)}}.
\end{align*}
Since the solution to \eqref{eq:capacitance_equation} is unique, it suffices to check that \eqref{eq:V_sol} satisfies the three conditions. Firstly, it is well known that \eqref{eq:V_sol} is a harmonic function with the appropriate behaviour in the far field \cite{jeffery1912form, lekner2011near, moon2012field, lim2014asymptotic}. To check the values on the boundaries $\D_1$, $\D_2$ we recall that \cite{lekner2011near, jeffery1912form}
\begin{equation}
1=\sqrt{2}\sqrt{\cosh\xi-\cos\theta}\sum_{n=0}^{\infty}e^{-(n+\frac{1}{2})|\xi|}P_n(\cos\theta),
\end{equation}
hence the subsitution of $\xi=-\xi_1$ and $\xi=\xi_2$ into \eqref{eq:V_sol} yields
\begin{align*}
V_1|_{\D_1}&=V_1(-\xi_1,\theta,\varphi)=\sqrt{2}\sqrt{\cosh\xi_1-\cos\theta}\sum_{n=0}^{\infty}e^{-(n+\frac{1}{2})\xi_1}P_n(\cos\theta)=1,\\
V_1|_{\D_2}&=V_1(\xi_2,\theta,\varphi)
=0,
\end{align*}
as well as similar results for $V_2$. Therefore, the solution to \eqref{eq:capacitance_equation} is given by \eqref{eq:V_sol}.

It remains to use the formula \eqref{eq:V_sol} for $V_j$ to calculate the capacitance coefficients through \eqref{eq:cap_reform}. We recall the identities \cite{moon2012field}
\begin{align}
\int_{-1}^1 \frac{P_n(s)}{\sqrt{\cosh\xi-s}}\de s &= \frac{2\sqrt{2}}{2n+1}e^{-(n+\frac{1}{2})|\xi|},\\
\int_{-1}^1 \frac{P_n(s)}{(\cosh\xi-s)^{3/2}}\de s &= \frac{2\sqrt{2}}{\sinh|\xi|}e^{-(n+\frac{1}{2})|\xi|},
\end{align}
from which we can show that
\begin{align}
\sqrt{2}\int_{\D_i} \partial_\nu \left(\sqrt{\cosh\xi-\cos\theta}\,e^{(n+\frac{1}{2})\xi}P_n(\cos\theta)\right)\de\sigma&=-8\pi \alpha \delta_{i2},\\
\sqrt{2}\int_{\D_i} \partial_\nu \left(\sqrt{\cosh\xi-\cos\theta}\,e^{-(n+\frac{1}{2})\xi}P_n(\cos\theta)\right)\de\sigma&=-8\pi \alpha \delta_{i1}.
\end{align}
Thus, integrating \eqref{eq:V_sol} over $\D_i$ gives
\begin{equation}
C_{ij}=8\pi \alpha \left(\delta_{i2} \sum_{n=0}^\infty A_n^j+ \delta_{i1} \sum_{n=0}^\infty B_n^j\right).
\end{equation}
\end{proof}

Using the results of \cite{lekner2011near}, we see from \Cref{lem:cap_coeffs} that the rescaled capacitance coefficients are given, at leading order, by
\begin{equation} \label{eq:cap_aymptotics}
\begin{split}
\tilde{C}_{11}&=\frac{3\alpha}{r_1^3(\xi_1+\xi_2)}\left[\log\left(\frac{2}{\xi_1+\xi_2}\right)-\psi\left(\frac{\xi_1}{\xi_1+\xi_2}\right)\right]+O(\sqrt{\epsilon}),\\
\tilde{C}_{22}&=\frac{3\alpha}{r_2^3(\xi_1+\xi_2)}\left[\log\left(\frac{2}{\xi_1+\xi_2}\right)-\psi\left(\frac{\xi_2}{\xi_1+\xi_2}\right)\right]+O(\sqrt{\epsilon}),\\
\tilde{C}_{12}&=-\frac{3\alpha}{r_1^3(\xi_1+\xi_2)}\left[\log\left(\frac{2}{\xi_1+\xi_2}\right)-\psi\left(1\right)\right]+O(\sqrt{\epsilon}),\\
\tilde{C}_{21}&=-\frac{3\alpha}{r_2^3(\xi_1+\xi_2)}\left[\log\left(\frac{2}{\xi_1+\xi_2}\right)-\psi\left(1\right)\right]+O(\sqrt{\epsilon}),
\end{split}
\end{equation}
where $\psi(z):=\dd{}{z}\log\Gamma(z)$ is the digamma function \cite{abramowitz1964handbook}, whose properties include $\psi(1)=-\gamma$ and $\psi(\frac{1}{2})=-\gamma-2\log2$. By combining \eqref{eq:cap_aymptotics} with \Cref{lem:res_asym} and the expression \eqref{eq:eigenvalues} we are able to find expressions for the resonant frequencies, at leading order.

\begin{theorem} \label{thm:asym_eigenvalues}
	The resonant frequencies of two spherical resonators with radii $r_1$, $r_2$ and separation distance $\epsilon$ are given by
	\begin{equation}
	\begin{split}
	\omega_1&\sim\sqrt{\delta},
	\\
	\omega_2&=\sqrt{\delta\frac{3 v_b^2}{2}\left(\frac{1}{r_1^3}+\frac{1}{r_2^3}\right)\frac{r_1r_2}{r_1+r_2}\log\left(\frac{2r_1r_2}{r_1+r_2}\frac{1}{\epsilon}\right)}+O\left(\sqrt{\delta}\right).
	\end{split}
	\end{equation}
	Again, the choice of $\epsilon\sim e^{-1/\delta^{1-\beta}}$, where $0<\beta<1$, means that as $\delta\to0$ we have that  $\omega_1\sim\sqrt{\delta}$ and $\omega_2\sim\delta^{\beta/2}$. 
\end{theorem}
\begin{proof}
We use a series expansion for the digamma function \cite{abramowitz1964handbook} to see that
\begin{equation*}
    \psi\left(\frac{\xi_i}{\xi_1+\xi_2}\right)=-\gamma-\sum_{n=1}^{\infty} \frac{z_i}{n(n-z_i)},
\end{equation*}
where $z_i=1-\xi_i/(\xi_1+\xi_2)$. Hence, we have that
\begin{equation} \label{eq:C_sig}
    \tilde{C}_{12}=-\tilde{C}_{11}+\sigma_1,\qquad
    \tilde{C}_{21}=-\tilde{C}_{22}+\sigma_2,
\end{equation}
where 
$$\sigma_i=\frac{3\alpha}{r_i^3(\xi_1+\xi_2)}\sum_{n=1}^{\infty} \frac{z_i}{n(n-z_i)}.$$
Note that $\sigma_i\sim1$ as $\delta\to0$. Therefore, the $\tilde{C}$ eigenvalues from \eqref{eq:eigenvalues} are given by
\begin{equation} \label{eq:evals}
\lambda_n=\frac{1}{2}\left(\C_{11}+\C_{22}+(-1)^n\sqrt{(\C_{11}+\C_{22})^2-4\C_{11}\sigma_2-4\C_{22}\sigma_1+4\sigma_1\sigma_2}\right).
\end{equation}
We can rewrite this as
\begin{equation} \label{eq:evals_expanded}
\lambda_n=\frac{1+(-1)^n}{2}(\C_{11}+\C_{22})+(-1)^{n+1}
\frac{\C_{11}\sigma_2+\C_{22}\sigma_1}{\C_{11}+\C_{22}}+O(\delta^{1-\beta}),
\end{equation}
where we have used the fact that the choice of $\epsilon$ relative to $\delta$ means that $(\C_{11}+\C_{22})^{-1}=O(\delta^{1-\beta})$. 

The formula for $\omega_2$ follows from \eqref{eq:evals_expanded} by using the leading order behaviour of $\tilde{C}_{ij}$, given in \eqref{eq:cap_aymptotics}, combined with the expansions
\begin{equation*}
    \xi_i=\frac{1}{r_i}\sqrt{\frac{2r_1r_2}{r_1+r_2}}\sqrt{\epsilon}+O(\epsilon^{3/2}),
    \qquad \frac{\alpha}{\xi_1+\xi_2}=\frac{r_1r_2}{r_1+r_2}+O(\epsilon).
\end{equation*}
In the case of $\omega_1$, the leading order term in \eqref{eq:evals_expanded} vanishes so the result follows from the fact that
\begin{equation}
\lambda_1=\frac{\C_{11}\sigma_2+\C_{22}\sigma_1}{\C_{11}+\C_{22}}+O(\delta^{1-\beta})
=\frac{r_1^3\sigma_1+r_2^3\sigma_2}{r_1^3+r_2^3}+O(\delta^{1-\beta})\sim1,
\end{equation}
as $\delta\to0$.
\end{proof}

\begin{remark}
We can see that the case of identical resonators \eqref{eq:res_identical} follows from the proof of \Cref{thm:asym_eigenvalues} since if $r_1=r_2=r$ then $\xi_1=\xi_2$ hence $\sigma_1=\sigma_2$ which means that \eqref{eq:evals} says that
\begin{equation}
\lambda_1=\sigma_1=\frac{3\alpha}{2r^3\xi_1}\sum_{n=1}^{\infty} \frac{1/2}{n(n-1/2)}=\frac{3}{r^2}\log2+O(\epsilon).
\end{equation}
\end{remark}

\section{Eigenmode gradient blow-up} \label{sec:eigenmodes}

We are interested in studying how the solution behaves in the region between the two spheres. The eigenmodes are known to be approximately constant on each resonator. If these constant values are different then, as the two resonators are moved close together, the gradient of the field between them will blow up. We wish to quantify the extent to which this happens.

Recall the decomposition \eqref{eq:psi_basis} which allows us to write the eigenmodes in terms of $\S_D[\psi_1]$ and $\S_D[\psi_2]$, as defined in \eqref{eq:psi_defs}. From the fact that the eigenvector of $\C$ associated to the eigenvalue $\lambda_n$ (as in \eqref{eq:eigenvalues}) is given by
\begin{equation} \label{eq:eigenvectors}
\left(\frac{\lambda_n-\C_{22}}{\C_{21}},\,1\,\right),
\end{equation}
we see that the eigenmodes are given, for $n=1,2$, by
\begin{equation} \label{eq:mode_defn}
    u_n(x)=\S_D[\phi_n](x)+O(\delta^{\beta/2}),
\end{equation}
where
\begin{equation} \label{eq:mode_densities}
    \phi_n:=\frac{\lambda_n-\C_{22}}{\C_{21}}\psi_1+\psi_2.
\end{equation}
By recalling the definition of the basis functions $\psi_1$ and $\psi_2$ \eqref{eq:psi_defs} we have that
\begin{equation} \label{eq:dvalue}
    u_n(x)=\begin{cases}
    \frac{\lambda_n-\C_{22}}{\C_{21}}+O(\delta^{\beta/2}), & x\in \D_1, \\
    1+O(\delta^{\beta/2}), & x\in \D_2.
    \end{cases}
\end{equation}
From the leading order behaviour of $\lambda_n$ \eqref{eq:evals_expanded} and of the capacitance coefficients \eqref{eq:cap_aymptotics} we have that, as $\delta\to0$,
\begin{equation} \label{eq:d1value}
    \frac{\lambda_n-\C_{22}}{\C_{21}}=\begin{cases}
    1+O(\delta^{1-\beta}), & n=1,\\
    -\frac{r_2^3}{r_1^3}+O(\delta^{1-\beta}), & n=2.
    \end{cases}
\end{equation}
Thus, we can show the following preliminary lemma.
\begin{lemma}
For sufficiently small $\delta>0$, $u_1|_{\D_1}$ and $u_1|_{\D_2}$ have the same sign whereas $u_2|_{\D_1}$ and $u_2|_{\D_2}$ have different signs.
\end{lemma}

Further to this, from \eqref{eq:dvalue} and \eqref{eq:d1value} we know that the eigenmodes converge to constant, non-zero values as $\delta\to0$. Since $\epsilon=\epsilon(\delta)$ is chosen so that $\epsilon\to0$ as $\delta\to0$, if the two leading order values are different then the maximum of the gradient of the solution between the two resonators must blow up as $\delta\to0$.

\begin{theorem} \label{thm:gradient_bounds}
Let $u_1$ and $u_2$ denote the subwavelength eigenmodes for two spherical resonators (with radii $r_1$ and $r_2$) separated by a distance $\epsilon$ which are normalised such that for any $x\in\D$
    \begin{equation*}
        \lim_{\delta\to0} |u_1(x)|\sim1,\qquad \lim_{\delta\to0} |u_2(x)|\sim1.
    \end{equation*}
    Suppose that the distance $\epsilon$ satisfies $\epsilon\sim e^{-1/\delta^{1-\beta}}$, then the maximal gradient of each eigenmode has the asymptotic behaviour, as $\delta\to0$,
    \begin{equation*}
    \max_{x\in\outside}|\nabla u_1(x)|\sim
    \begin{cases}
    1, & \text{if } r_1=r_2,\\
    \frac{1}{\epsilon|\log\epsilon|}, & \text{otherwise},
    \end{cases}
\end{equation*}
    and
\begin{equation*}
    \max_{x\in\outside}|\nabla u_2(x)|\sim \frac{1}{\epsilon}.
\end{equation*}
\end{theorem}
\begin{proof} We first remark that the desired normalisation of the eigenmodes is possible thanks to \eqref{eq:mode_defn}-\eqref{eq:d1value}. We prove the desired behaviour by decomposing the leading order expressions for the eigenmodes into two functions. The first, which does not have a singular gradient as $\epsilon\to0$, is defined as the solution to
\begin{equation} \label{eq:nonsingular_part}
\begin{cases}
\Delta h_1 = 0, & \text{in } \outside, \\
h_1=1, & \text{on } \D_1\cup\D_2,\\
h_1(x)=O\left(\tfrac{1}{|x|}\right), & \text{as } |x|\to\infty.
\end{cases}
\end{equation}
The fact that $\nabla h_1$ is bounded as $\epsilon\to0$ follows from the fact that $h_1|_{\D_1}=h_1|_{\D_2}$, \emph{e.g.} from Lemma~2.3 of \cite{bao2009gradient} or by applying the result of \cite{ammari2007estimates}.

For the singular part, we use a function that has been used in other settings, defined as the solution to
\begin{equation} \label{eq:singular_part}
\begin{cases}
\Delta h_2 = 0, & \text{in } \outside, \\
h_2=c_i, & \text{on } \D_i,\\
h_2(x)=O\left(\tfrac{1}{|x|}\right), & \text{as } |x|\to\infty,\\
\int_{\D_i} \ddp{h_2}{\nu}\big|_+ \de\sigma =(-1)^i,
\end{cases}
\end{equation}
for some constants $c_i$. We know, \emph{e.g.} from Theorems~1.1 and 1.2 of \cite{bao2009gradient} or from Proposition~5.3 of \cite{lim2014asymptotic} that
\begin{equation}
    \max_{x\in\outside} |\nabla h_2|\sim\frac{1}{\epsilon|\log\epsilon|} \text{ as } \delta\to0.
\end{equation}

We now wish to write the leading order term of \eqref{eq:mode_defn} in terms of $h_1$ and $h_2$, that is find $A_n$ and $B_n$ such that for all $x\in\outside$
\begin{equation} \label{eq:decomp_sing}
    \S_D[\phi_n](x)=\frac{\lambda_n-\C_{22}}{\C_{21}}\S_D[\psi_1](x)+\S_D[\psi_2](x)=A_n h_1(x) + B_n h_2(x),
\end{equation}
where $A_n$ and $B_n$ are constant with respect to $x$, but may depend on $\epsilon$. Differentiating \eqref{eq:decomp_sing} and integrating over $\D_1$ and $\D_2$, respectively, gives the equations
\begin{align}
    \left(\frac{\lambda_n-\C_{22}}{\C_{21}}-1\right)\C_{11}+\sigma_1&=A_n\sigma_1+B_n, \label{eq:D1}\\
    \lambda_n&=A_n\sigma_2-B_n,\label{eq:D2}
\end{align}
where we have used the fact that $h_1=\S_D[\psi_1+\psi_2]$, the representation \eqref{eq:cap_reform} for the capacitance coefficients and the notation $\sigma_i$ from \eqref{eq:C_sig}.

We can solve \eqref{eq:D1} and \eqref{eq:D2} for $A_n$ and $B_n$. We see, firstly, that 
\begin{equation} \label{eq:A_equation}
    \left(\frac{\lambda_n-\C_{22}}{\C_{21}}-1\right)\C_{11}+\sigma_1+ \lambda_n=A_n(\sigma_1+\sigma_2).
\end{equation}
From which, we can use \eqref{eq:d1value} as well as the fact that $\lambda_1=O(1)$ and $\sigma_1=O(1)$ to see that
\begin{equation} \label{eq:A1}
    A_1=O(1) \text{ as } \delta\to0.
\end{equation}
For the case where $n=2$, we can additionally use \eqref{eq:evals_expanded} to see that the left-hand side of \eqref{eq:A_equation} is given by
\begin{equation}
    \left(\frac{\lambda_2-\C_{22}}{\C_{21}}-1\right)\C_{11}+\sigma_1+ \lambda_2=-\frac{r_2^3}{r_1^3}\C_{11}+\C_{22}+O(1),
\end{equation}
thus, we have that
\begin{equation} \label{eq:A2_asym}
    r_1\neq r_2 \implies A_2\sim|\log\epsilon| \text{ as } \delta\to0.
\end{equation}
Conversely, if $r_1=r_2$ then $\lambda_2=\C_{22}-\C_{21}$ and hence
\begin{equation}
    \left(\frac{\lambda_2-\C_{22}}{\C_{21}}-1\right)\C_{11}+\sigma_1+ \lambda_2=-2\C_{22}+\sigma_2+\C_{22}-\C_{21}=0,
\end{equation}
so \eqref{eq:A_equation} gives that
\begin{equation} \label{eq:A2_sym}
    r_1= r_2 \implies A_2=0 \text{ as } \delta\to0.
\end{equation}

We can now use \eqref{eq:D2} to find $B_n$. The behaviour of $B_1$ is similar to that of $A_2$ in the sense that if $r_1=r_2$ then $\lambda_1=\sigma_1=\sigma_2$ and $A_1=1$ so \eqref{eq:D2} gives that
\begin{equation} \label{eq:B1_sym}
    r_1= r_2 \implies B_1=0 \text{ as } \delta\to0,
\end{equation}
whereas
\begin{equation}
    r_1\neq r_2 \implies B_1\sim1 \text{ as } \delta\to0.
\end{equation}
The case of $B_2$ is much simpler, since we always have that
\begin{equation}
    B_2\sim|\log\epsilon| \text{ as } \delta\to0.
\end{equation}

Finally, the result follows by combining the above results, namely the behaviour of the coefficients $A_n$ and $B_n$ and the estimates for $\nabla h_1$ and $\nabla h_2$.
\end{proof}

\section{Scattered solution} \label{sec:scattering}

We now wish to study the scattered field in response to an incoming plane wave $u^{in}$, writing the solution in terms of the subwavelength eigenmodes studied above.

\begin{theorem} \label{thm:scattering}
    Let $u_1$ and $u_2$ be the two subwavelength eigenmodes, normalised according to \eqref{eq:mode_defn} and \eqref{eq:mode_densities}. Then, the solution $u$ to the scattering problem \eqref{eq:helmholtz_equation} with incoming plane wave $u^{in}$ with frequency $\omega$ is given, for $x\in\outside$, by
    \begin{equation*}
        u(x)=u^{in}(x)-\S_D^k\left[\S_D^{-1}[u^{in}]\right](x)+au_1(x)+bu_2(x)+O(\omega),
    \end{equation*}
    where the coefficients $a$ and $b$ are given, as $\delta,\omega\to0$, by
    \begin{align*}
        a&= \frac{\delta}{\omega^2-\omega_1^2}\frac{v_b^2}{|D|}\int_{\D} \S_D^{-1}[u^{in}] \de\sigma+O(\delta^{2-\beta}+\delta^{1-\beta}\omega^2+\omega^3),\\
        b&= -\frac{\delta}{\omega^2-\omega_2^2}\frac{v_b^2}{|D|}\left(\int_{\D_1} \S_D^{-1}[u^{in}] \de\sigma-\frac{|D_1|}{|D_2|}\int_{\D_2} \S_D^{-1}[u^{in}] \de\sigma \right) + O(\delta^{2-\beta}+\delta^{1-\beta}\omega^2+\omega^3).
    \end{align*}
\end{theorem}
\begin{proof}
If $(\phi,\psi)$ solves the scattering problem \eqref{eq:A_matrix_equation} then using the asymptotic expansions \eqref{eq:S_series} and \eqref{eq:K_series} we see that
\begin{align}
\S_D[\phi-\psi]&=u^{in}+O(\omega),\label{eq:61}\\
\left(-\frac{1}{2}I+\K_D^*+k_b^2\K_{D,2}\right)[\phi]-\delta\left(\frac{1}{2}I+\K_D^*\right)[\psi]&=O(\delta\omega+\omega^3).\label{eq:62}
\end{align}
From \eqref{eq:61}, we know that
\begin{equation}
    \psi=\phi-\S_D^{-1}[u^{in}]+O(\omega),
\end{equation}
so are able to write that
\begin{equation} \label{eq:4.6}
    \left(-\frac{1}{2}I+\K_D^*\right)[\phi]+k_b^2\K_{D,2}[\phi]-\delta\left(\frac{1}{2}I+\K_D^*\right)[\phi]=-\delta\left(\frac{1}{2}I+\K_D^*\right)\S_D^{-1}[u^{in}]
    +O(\delta\omega+\omega^3).
\end{equation}

We can make the decomposition
\begin{equation} \label{eq:decomp}
\phi=a\phi_1+b\phi_2+\phi_3,
\end{equation}
for constants $a,b=O(1)$, where $\phi_1$ and $\phi_2$ are the densities corresponding to the two subwavelength eigenmodes, defined in \eqref{eq:mode_densities}, and $\phi_3\in L^2(\D)$ is orthogonal to both $\phi_1$ and $\phi_2$ in $L^2(\D)$. We can see that 
$\|\phi_3\|_{L^2(\D)}=O(\delta+\omega^2)$ (\emph{cf.} Theorem~4.2 of \cite{ammari2017double}).

If we use the decomposition \eqref{eq:decomp} and integrate \eqref{eq:4.6} over $\D$, then the properties \eqref{k_properties} give us the equation
\begin{align} \label{eq:Dint}
-k_b^2\int_{D}\S_D[a\phi_1+b\phi_2]\de x-\delta\int_{\D}a\phi_1+b\phi_2\de\sigma=-\delta\int_{\D}\S_{D}^{-1}[u^{in}]\de\sigma +O(\delta\omega+\omega^3).
\end{align}

Recall that $\phi_1$ and $\phi_2$ are defined such that \eqref{eq:313} is satisfied exactly when $\omega$ is equal to the corresponding resonant frequency. Therefore, we have that
\begin{align} \label{eq:resonance_property}
-\frac{a\omega_1^2}{v_b^2}\int_{D}\S_D[\phi_1]\de x-\frac{b\omega_2^2}{v_b^2}\int_{D}\S_D[\phi_2]\de x-\delta\int_{\D}a\phi_1+b\phi_2\de\sigma=O(\delta\omega+\omega^3).
\end{align}
From \eqref{eq:d1value} we can show that
\begin{equation} 
    \int_D \S_D[\phi_n]\de x=\frac{\lambda_n-\C_{22}}{\C_{21}}|D_1|+|D_2|=\begin{cases} |D|+O(\delta^{1-\beta}),&n=1,\\ O(\delta^{1-\beta}), & n=2.\end{cases}
\end{equation}
Then, subtracting \eqref{eq:resonance_property} from \eqref{eq:Dint} we reach
\begin{equation}
    a\frac{\omega_1^2-\omega^2}{v_b^2}|D|
    = -\delta
    \int_{\D}\S_{D}^{-1}[u^{in}]\de\sigma +O(\delta^{2-\beta}+\delta^{1-\beta}\omega^2+\omega^3), 
\end{equation}
which can be solved to give the formula for $a$. The formula for $b$ can be found by repeating these steps but instead integrating \eqref{eq:4.6} over $\D_1-\D_2$ and using the fact that
\begin{equation} 
    \int_{D_1-D_2} \S_D[\phi_n]\de x=\frac{\lambda_n-\C_{22}}{\C_{21}}|D_1|-|D_2|=\begin{cases} |D_1|-|D_2|+O(\delta^{1-\beta}),&n=1,\\ -2|D_2|+O(\delta^{1-\beta}), & n=2.\end{cases}
\end{equation}
This gives the equation
\begin{equation}
    a\frac{\omega_1^2-\omega^2}{v_b^2}(|D_1|-|D_2|)+b\frac{\omega_2^2-\omega^2}{v_b^2}(-2|D_2|) = -\delta 
    \int_{\D_1-\D_2} \S_{D}^{-1}[u^{in}] \de\sigma +O(\delta^{2-\beta}+\delta^{1-\beta}\omega^2+\omega^3),
\end{equation}
which can be solved to give the formula for $b$.
\end{proof}

\begin{remark}
It is also important to understand how the term $\S_D^k\left[\S_D^{-1}[u^{in}]\right](x)$ behaves, for $x\in\outside$, as $\epsilon\to0$. We have that
\begin{equation*}
\S_D^k\left[\S_D^{-1}[u^{in}]\right](x)=\S_D\left[\S_D^{-1}[u^{in}(0)]\right](x)+O(\omega),
\end{equation*}
and are able to write that $\S_D\left[\S_D^{-1}[u^{in}(0)]\right]=u^{in}(0)(V_1+V_2)$, as defined in \eqref{eq:V_sol}. From which we can show, in particular, that $\S_D^k\left[\S_D^{-1}[u^{in}]\right](x)$ is bounded as $\epsilon\to0$.
\end{remark}

\section{Concluding remarks}

Structures composed of subwavelength resonators have been shown to have remarkable wave-guiding abilities. In this paper, we have conducted an asymptotic analysis of the behaviour of two subwavelength resonators that are close to touching. We have shown that the two subwavelength resonant frequencies have different asymptotic behaviour and have derived estimates for the rate at which the gradient of each eigenmode blows up, accounting for the differences between symmetric and non-symmetric structures.

We have studied the case of spherical resonators in this work, but this could be generalised to shapes that are strictly convex in a region of the close-to-touching points. This relies on using spheres with the same curvature to approximate the structure, as has been done in the setting of antiplane elasticity \cite{ammari2013spectral} and full linear elasticity \cite{kang2019quantitative}.

Understanding the different asymptotic behaviour of the two eigenfrequencies is useful if one wants to design structures for specific applications. For example, one might want to construct an array that responds to a specific range of frequencies \cite{davies2019fully, davies2020hopf} or a structure that has subwavelength band gaps \cite{ammari2019topological}. In addition, the estimates for the blow-up of the gradient of the eigenmodes are valuable since the gradient of the acoustic pressure describes the forces that the resonators exert on one another in the presence of sound waves. Known as the secondary Bjerknes forces \cite{bjerknes1906fields, crum1975bjerknes, lanoy2015manipulating, mettin1997bjerknes, pandey2019asymmetricity}, this work provides an approach to understanding these forces in the case of close-to-touching bubbles.

\section*{Acknowledgement}

We are grateful to Erik Orvehed Hiltunen for their insightful comments during discussions about this work.

\bibliographystyle{abbrv}
\bibliography{/scratch/users/bdavies/Documents/myacousticbib.bib}

\begin{thebibliography}{10}

\bibitem{abramowitz1964handbook}
M.~Abramowitz and I.~A. Stegun.
\newblock {\em Handbook of mathematical functions with formulas, graphs, and
  mathematical tables}.
\newblock Nat. Bur. Stand., Washington D.C., 1964.

\bibitem{ammari2013spectral}
H.~Ammari, G.~Ciraolo, H.~Kang, H.~Lee, and K.~Yun.
\newblock Spectral analysis of the {Neumann--Poincar{\'e}} operator and
  characterization of the stress concentration in anti-plane elasticity.
\newblock {\em Arch. Rational Mech. An.}, 208(1):275--304, 2013.

\bibitem{ammari2007estimates}
H.~Ammari, G.~Dassios, H.~Kang, and M.~Lim.
\newblock Estimates for the electric field in the presence of adjacent
  perfectly conducting spheres.
\newblock {\em Q. Appl. Math.}, 65(2):339--355, 2007.

\bibitem{davies2019fully}
H.~Ammari and B.~Davies.
\newblock A fully-coupled subwavelength resonance approach to filtering
  auditory signals.
\newblock {\em Proc. R. Soc. A}, 475(2228):20190049, 2019.

\bibitem{davies2020hopf}
H.~Ammari and B.~Davies.
\newblock Mimicking the active cochlea with a fluid-coupled array of
  subwavelength {Hopf} resonators.
\newblock {\em Proc. R. Soc. A}, 476(2234):20190870, 2020.

\bibitem{ammari2019topological}
H.~Ammari, B.~Davies, E.~O. Hiltunen, and S.~Yu.
\newblock Topologically protected edge modes in one-dimensional chains of
  subwavelength resonators.
\newblock {\em J. Math. Pures Appl.}, (to appear), 2020.

\bibitem{ammari2017sub}
H.~Ammari, B.~Fitzpatrick, D.~Gontier, H.~Lee, and H.~Zhang.
\newblock Sub-wavelength focusing of acoustic waves in bubbly media.
\newblock {\em Proc. R. Soc. A}, 473(2208):20170469, 2017.

\bibitem{ammari2018minnaert}
H.~Ammari, B.~Fitzpatrick, D.~Gontier, H.~Lee, and H.~Zhang.
\newblock Minnaert resonances for acoustic waves in bubbly media.
\newblock {\em Ann. I. H. Poincar{\'e}--A. N.}, 35(7):1975--1998, 2018.

\bibitem{ammari2018mathematical}
H.~Ammari, B.~Fitzpatrick, H.~Kang, M.~Ruiz, S.~Yu, and H.~Zhang.
\newblock {\em Mathematical and computational methods in photonics and
  phononics}, volume 235 of {\em Mathematical surveys and monographs}.
\newblock American Mathematical Society, Providence, 2018.

\bibitem{ammari2017double}
H.~Ammari, B.~Fitzpatrick, H.~Lee, S.~Yu, and H.~Zhang.
\newblock Double-negative acoustic metamaterials.
\newblock {\em Quart. Appl. Math.}, 77(4):767--791, 2019.

\bibitem{ammari2004boundary}
H.~Ammari and H.~Kang.
\newblock Boundary layer techniques for solving the {Helmholtz} equation in the
  presence of small inhomogeneities.
\newblock {\em J. Math. Anal. Appl.}, 296(1):190--208, 2004.

\bibitem{ammari2009layer}
H.~Ammari, H.~Kang, and H.~Lee.
\newblock {\em Layer potential techniques in spectral analysis}, volume 153 of
  {\em Mathematical Surveys and Monographs}.
\newblock American Mathematical Society, Providence, 2009.

\bibitem{ammari2007optimal}
H.~Ammari, H.~Kang, H.~Lee, J.~Lee, and M.~Lim.
\newblock Optimal estimates for the electric field in two dimensions.
\newblock {\em J. Math. Pures Appl.}, 88(4):307--324, 2007.

\bibitem{ammari2019shape}
H.~Ammari, M.~Putinar, M.~Ruiz, S.~Yu, and H.~Zhang.
\newblock Shape reconstruction of nanoparticles from their associated plasmonic
  resonances.
\newblock {\em J. Math. Pure. Appl.}, 122:23--48, 2019.

\bibitem{bao2009gradient}
E.~S. Bao, Y.~Y. Li, and B.~Yin.
\newblock Gradient estimates for the perfect conductivity problem.
\newblock {\em Arch. Rational Mech. Anal.}, 193(1):195--226, 2009.

\bibitem{bao2015gradient}
J.~Bao, H.~Li, and Y.~Li.
\newblock Gradient estimates for solutions of the lam{\'e} system with
  partially infinite coefficients.
\newblock {\em Arch. Rational Mech. Anal.}, 215(1):307--351, 2015.

\bibitem{bao2017gradient}
J.~Bao, H.~Li, and Y.~Li.
\newblock Gradient estimates for solutions of the lam{\'e} system with
  partially infinite coefficients in dimensions greater than two.
\newblock {\em Adv. Math.}, 305:298--338, 2017.

\bibitem{bjerknes1906fields}
V.~F.~K. Bjerknes.
\newblock {\em Fields of force}.
\newblock The Columbia University Press, New York, 1906.

\bibitem{bonnetier2013spectrum}
E.~Bonnetier and F.~Triki.
\newblock On the spectrum of the poincar{\'e} variational problem for two
  close-to-touching inclusions in 2d.
\newblock {\em Arch. Rational Mech. Anal.}, 209(2):541--567, 2013.

\bibitem{colton1983integral}
D.~Colton and R.~Kress.
\newblock {\em Integral equation methods in scattering theory}.
\newblock Wiley, New York, 1983.

\bibitem{crum1975bjerknes}
L.~A. Crum.
\newblock Bjerknes forces on bubbles in a stationary sound field.
\newblock {\em J. Acoust. Soc. Am.}, 57(6):1363--1370, 1975.

\bibitem{devaud2008minnaert}
M.~Devaud, T.~Hocquet, J.-C. Bacri, and V.~Leroy.
\newblock The minnaert bubble: an acoustic approach.
\newblock {\em Eur. J. Phys.}, 29(6):1263, 2008.

\bibitem{gohberg2009holomorphic}
I.~Gohberg and J.~Leiterer.
\newblock {\em Holomorphic operator functions of one variable and applications:
  methods from complex analysis in several variables}, volume 192 of {\em
  Operator Theory Advances and Applications}.
\newblock Birkhäuser, Basel, 2009.

\bibitem{gorb2015singular}
Y.~Gorb.
\newblock Singular behavior of electric field of high-contrast concentrated
  composites.
\newblock {\em Multiscale Model. Sim.}, 13(4):1312--1326, 2015.

\bibitem{hooshmand2019collective}
N.~Hooshmand and M.~A. El-Sayed.
\newblock Collective multipole oscillations direct the plasmonic coupling at
  the nanojunction interfaces.
\newblock {\em P. Natl. Acad. Sci. USA}, 116(39):19299--19304, 2019.

\bibitem{jeffery1912form}
G.~B. Jeffery.
\newblock On a form of the solution of laplace's equation suitable for problems
  relating to two spheres.
\newblock {\em P. Roy. Soc. Lond. A Mat.}, 87(593):109--120, 1912.

\bibitem{kang2019quantitative}
H.~Kang and S.~Yu.
\newblock Quantitative characterization of stress concentration in the presence
  of closely spaced hard inclusions in two-dimensional linear elasticity.
\newblock {\em Arch. Rational Mech. Anal.}, 232:121--196, 2019.

\bibitem{khattak2019linking}
H.~K. Khattak, P.~Bianucci, and A.~D. Slepkov.
\newblock Linking plasma formation in grapes to microwave resonances of aqueous
  dimers.
\newblock {\em P. Natl. Acad. Sci. USA}, 116(10):4000--4005, 2019.

\bibitem{kim2019electric}
J.~Kim and M.~Lim.
\newblock Electric field concentration in the presence of an inclusion with
  eccentric core-shell geometry.
\newblock {\em Math. Ann.}, 373(1--2):517--551, 2019.

\bibitem{kushwaha1998sound}
M.~Kushwaha, B.~Djafari-Rouhani, and L.~Dobrzynski.
\newblock Sound isolation from cubic arrays of air bubbles in water.
\newblock {\em Phys. Lett. A}, 248(2-4):252--256, 1998.

\bibitem{lanoy2015manipulating}
M.~Lanoy, C.~Derec, A.~Tourin, and V.~Leroy.
\newblock Manipulating bubbles with secondary bjerknes forces.
\newblock {\em Appl. Phys. Lett.}, 107(21):214101, 2015.

\bibitem{lekner2011near}
J.~Lekner.
\newblock Near approach of two conducting spheres: Enhancement of external
  electric field.
\newblock {\em J. Electrostat.}, 69(6):559--563, 2011.

\bibitem{leroy2009design}
V.~Leroy, A.~Bretagne, M.~Fink, H.~Willaime, P.~Tabeling, and A.~Tourin.
\newblock Design and characterization of bubble phononic crystals.
\newblock {\em Appl. Phys. Lett.}, 95(17):171904, 2009.

\bibitem{lim2014asymptotic}
M.~Lim and S.~Yu.
\newblock Asymptotic analysis for superfocusing of the electric field in
  between two nearly touching metallic spheres.
\newblock {\em arXiv preprint arXiv:1412.2464}, 2014.

\bibitem{lim2014asymptotic_elas}
M.~Lim and S.~Yu.
\newblock Stress concentration for two nearly touching circular holes.
\newblock {\em arXiv preprint arXiv:1705.10400}, 2017.

\bibitem{lim2009blow}
M.~Lim and K.~Yun.
\newblock Blow-up of electric fields between closely spaced spherical perfect
  conductors.
\newblock {\em Commun. Part. Diff. Eq.}, 34(10):1287--1315, 2009.

\bibitem{mcphedran1981electrostatic}
R.~McPhedran and W.~Perrins.
\newblock Electrostatic and optical resonances of cylinder pairs.
\newblock {\em Appl. Phys.}, 24(4):311--318, 1981.

\bibitem{mcphedran1988asymptotic}
R.~McPhedran, L.~Poladian, and G.~W. Milton.
\newblock Asymptotic studies of closely spaced, highly conducting cylinders.
\newblock {\em P. Roy. Soc. Lond. A Mat.}, 415(1848):185--196, 1988.

\bibitem{mettin1997bjerknes}
R.~Mettin, I.~Akhatov, U.~Parlitz, C.~Ohl, and W.~Lauterborn.
\newblock Bjerknes forces between small cavitation bubbles in a strong acoustic
  field.
\newblock {\em Phys. Rev. E}, 56(3):2924, 1997.

\bibitem{minnaert1933musical}
M.~Minnaert.
\newblock On musical air-bubbles and the sounds of running water.
\newblock {\em Philos. Mag.}, 16(104):235--248, 1933.

\bibitem{moon2012field}
P.~Moon and D.~E. Spencer.
\newblock {\em Field theory handbook: including coordinate systems,
  differential equations and their solutions}.
\newblock Springer, Berlin, 1971.

\bibitem{pandey2019asymmetricity}
V.~Pandey.
\newblock Asymmetricity and sign reversal of secondary bjerknes force from
  strong nonlinear coupling in cavitation bubble pairs.
\newblock {\em Phys. Rev. E}, 99(4):042209, 2019.

\bibitem{pendry2012transformation}
J.~Pendry, A.~Aubry, D.~Smith, and S.~Maier.
\newblock Transformation optics and subwavelength control of light.
\newblock {\em Science}, 337(6094):549--552, 2012.

\bibitem{poladian1989asymptotic}
L.~Poladian.
\newblock Asymptotic behaviour of the effective dielectric constants of
  composite materials.
\newblock {\em P. Roy. Soc. Lond. A Mat.}, 426(1871):343--359, 1989.

\bibitem{romero2006plasmons}
I.~Romero, J.~Aizpurua, G.~W. Bryant, and F.~J.~G. De~Abajo.
\newblock Plasmons in nearly touching metallic nanoparticles: singular response
  in the limit of touching dimers.
\newblock {\em Opt. express}, 14(21):9988--9999, 2006.

\bibitem{yu2018plasmonic}
S.~Yu and H.~Ammari.
\newblock Plasmonic interaction between nanospheres.
\newblock {\em SIAM Rev.}, 60(2):356--385, 2018.

\bibitem{yu2019hybridization}
S.~Yu and H.~Ammari.
\newblock Hybridization of singular plasmons via transformation optics.
\newblock {\em P. Natl. Acad. Sci. USA}, 116(28):13785--13790, 2019.

\bibitem{yun2007estimates}
K.~Yun.
\newblock Estimates for electric fields blown up between closely adjacent
  conductors with arbitrary shape.
\newblock {\em SIAM J. Appl. Math.}, 67(3):714--730, 2007.

\end{thebibliography}
\end{document}